\documentclass[10pt,psamsfonts]{amsart}
\usepackage[utf8]{inputenc}

\usepackage[makeroom]{cancel}
\usepackage{float}
\usepackage{mathabx}
\usepackage{amsmath}
\usepackage{amsthm}
\usepackage{amssymb}
\usepackage{amscd}
\usepackage{amsfonts}
\usepackage{amsbsy}
\usepackage{graphicx}
\usepackage[dvips]{psfrag}
\usepackage{array}
\usepackage{xcolor}
\usepackage{epsfig}
\usepackage{url}
\usepackage{hyperref}
\usepackage{enumitem}
\usepackage{overpic}
\usepackage{epstopdf}
\usepackage{apptools}

\newcommand{\Int}{\mathrm{Int}}

\newcommand{\R}{\ensuremath{\mathbb{R}}}

\newcommand{\ov}{\overline}

\newcommand{\x}{\mathbf{x}}

\newcommand{\sgn}{\mathrm{sign}}

\newcommand{\bb}{\mathbf{b}}

\def\p{\partial}
\def\e{\varepsilon}

\newtheorem {theorem} {Theorem}

\newtheorem {proposition} {Proposition}
\newtheorem {corollary} {Corollary}
\newtheorem {lemma}[theorem] {Lemma}

\newtheorem {remark} {Remark}

\usepackage{mathpazo}

\begin{document}
\renewcommand{\arraystretch}{1.5}

\title[Bounding the number of limit cycles in piecewise linear differential systems]
{Bounding the number of limit cycles\\ in piecewise linear differential systems: \\ methodology and worked examples}

\author[V. Carmona, F. Fernandez-S\'{a}nchez, and D. D. Novaes]
{Victoriano Carmona$^1$, Fernando Fern\'{a}ndez-S\'{a}nchez$^2$,\\ and Douglas D. Novaes$^3$}

\address{$^1$ Dpto. Matem\'{a}tica Aplicada II \& IMUS, Universidad de Sevilla, Escuela Polit\'ecnica Superior.
Calle Virgen de \'Africa 7, 41011 Sevilla, Spain.} 
 \email{vcarmona@us.es} 

\address{$^2$ Dpto. Matem\'{a}tica Aplicada II \& IMUS, Universidad de Sevilla, Escuela T\'{e}cnica Superior de Ingenier\'{i}a.
Camino de los Descubrimientos s/n, 41092 Sevilla, Spain.} \email{fefesan@us.es}

\address{$^3$ Departamento de Matem\'{a}tica, Instituto de Matem\'{a}tica, Estatística e Computa\c{c}\~{a}o Cient\'{i}fica (IMECC), Universidade
Estadual de Campinas (UNICAMP), Rua S\'{e}rgio Buarque de Holanda, 651, Cidade Universit\'{a}ria Zeferino Vaz, 13083--859, Campinas, SP,
Brazil.} \email{ddnovaes@unicamp.br} 



\maketitle

\begin{abstract}
In 2012, Huan and Yang introduced the first piecewise linear differential system with two zones separated by a straight line having at least three limit cycles, serving as a counterexample to the Han-Zhang conjecture that said that such systems have no more than two limit cycles. Over the past decade, extensive research has been conducted to explore periodic solutions in piecewise linear differential systems. However, the question of whether the Huan-Yang example indeed has exactly three limit cycles has remained unresolved, primarily due to the lack of techniques for bounding the number of limit cycles in these systems. Based on the authors' recent results, this paper presents a methodology for bounding the number of limit cycles in piecewise linear systems. This methodology conclusively establishes that the Huan-Yang example has exactly three limit cycles. Our methodology has a broader applicability and constitutes a powerful tool for analyzing and bounding the number of limit cycles in any explicit example of piecewise linear differential systems. We extend our analysis to several recent examples of significance in the literature and show that they also exhibit exactly three limit cycles. Finally, we present an algebraic criterion for bounding the number of crossing limit cycles in the focus-focus case.

\end{abstract}

\section{Introduction}

In the last decades, many studies investigated the number of limit cycles in piecewise linear differential systems of kind:
\begin{equation}\label{s1}
\dot \x =
\left\{\begin{array}{l}
A_L\x+\bb_L, \quad\textrm{if}\quad x_1\leq 0,\\
A_R\x+\bb_R, \quad\textrm{if}\quad x_1\geq 0,
\end{array}\right.
\end{equation}
being $\x=(x_1,x_2)\in\R^2$, $A_{L,R}=(a_{ij}^{L,R})_{2\times 2}$, and $\bb_{L,R}=(b_1^{L,R},b_2^{L,R})\in\R^2$, which has their solutions established by Filippov's convention \cite{Filippov88}.

Starting with Lum and Chua \cite{LumChua91} in 1991, they established the conjecture that, under the continuity conditions $a_{12}^L = a_{12}^R$, $a_{22}^L = a_{22}^R$, and $\bb_L = \bb_R$, system \eqref{s1} would possess at most a unique limit cycle. A first proof for this conjecture was provided seven years later, in 1998, by Freire et al. \cite{FreireEtAl98}. Subsequent research aimed to relax these continuity conditions to explore more complex dynamics. Han and Zhang \cite{HanZhang10} provided in 2010  the first examples of differential systems of type \eqref{s1} presenting two limit cycles. Their investigation led them to conjecture that these systems could not have more than two limit cycles. Nevertheless, two years later, Huan and Yang \cite{HuanYang12} provided a counterexample for this conjecture, presenting the first example of a system of type \eqref{s1} with at least three limit cycles. Their example assumed:
\begin{equation}\label{HY-example}
\begin{aligned}
&A_R=\left (\begin{array}{cc} \dfrac{19}{500}&-\dfrac{1}{10}\vspace{0.3cm}\\ \dfrac{1}{10}& \dfrac{19}{500}\end{array}\right),\quad A_L=\left (\begin{array}{cc} 1&-5\vspace{0.3cm}\\ \dfrac{377}{1000}& -\dfrac{13}{10}\end{array}\right),\\ 
&\bb_R=\left (\begin{array}{c} \dfrac{19}{500}\vspace{0.3cm} \\ \dfrac{1}{10}\end{array}\right),\quad \text{and}\quad \bb_L=\left (\begin{array}{cc}1\vspace{0.3cm}\\ \dfrac{377}{1000}\end{array}\right).
\end{aligned}
\end{equation}
Huan and Yang employed a numerical approach to observe the three limit cycles. However, Llibre and Ponce \cite{LlibrePonce12} conclusively demonstrated, in the same year, the presence of these numerically observed limit cycles by applying the Newton-Kantorovich theorem. Ever since, several studies obtained systems of type \eqref{s1} for which the presence of three limit cycles is ensured, we may refer to \cite{BuzziEtAl13,cardoso20,FreireEtAl14, FreireEtAl14b, LlibreEtAl15b,NovaesTorregrosa17}. In what follows we emphasize two other  explicit examples.

In \cite{FPTT21}, Freire et al. considered differential systems of type \eqref{s1} with a periodic orbit at infinity. By means of Hopf bifurcation at infinity they derived an example with three limit cycles of high amplitude. Their example assumed:
\begin{equation}\label{FPTT-example}
\begin{aligned}
&A_L=\left (\begin{array}{cc} -\dfrac{1}{4}&-1\vspace{0.3cm}\\ \dfrac{65}{64}& 0\end{array}\right),\quad A_R=\left (\begin{array}{cc} \dfrac{3276710}{13106841}&-1\vspace{0.3cm}\\ \dfrac{174473488105306}{171789280999281}& 0\end{array}\right),\\ 
&\bb_L=\left (\begin{array}{c} \dfrac{260534}{1045519}\vspace{0.3cm} \\ -\dfrac{65}{64}\end{array}\right),\quad \text{and}\quad \bb_R=\left (\begin{array}{cc}-\dfrac{260534}{1045519}\vspace{0.3cm}\\ -\dfrac{96440395023695996806}{95571015330487000887}\end{array}\right).
\end{aligned}
\end{equation}

The most recent example of system of type \eqref{s1} with three limit cycles has been provided by Gasull et al. \cite{rondon} in the context of piecewise holomorphic system. Their example assumed:
\begin{equation}\label{GRS-example}
\begin{aligned}
&A_L=\left (\begin{array}{cc} -\dfrac{1}{5}&1\vspace{0.3cm}\\ -1& -\dfrac{1}{5}\end{array}\right),\quad A_R=\left (\begin{array}{cc} \dfrac{3}{8}&1\vspace{0.3cm}\\ -1& \dfrac{3}{8}\end{array}\right),\\ 
&\bb_L=\left (\begin{array}{c} \dfrac{1}{250}\vspace{0.3cm} \\ -\dfrac{51}{50}\end{array}\right),\quad \text{and}\quad \bb_R=\left (\begin{array}{cc}\dfrac{1159}{1000}\vspace{0.3cm}\\ -\dfrac{14333}{2000}\end{array}\right).
\end{aligned}
\end{equation}

Due to the lack of techniques for bounding the number of limit cycles in differential systems of type \eqref{s1}, it remained an open question for over a decade whether the Huan-Yang example, i.e., system \eqref{s1} with parameter values specified in \eqref{HY-example}, could have more than three limit cycles. The same situation holds for the other two mentioned examples, when the parameter values are given either by  \eqref{FPTT-example} or \eqref{GRS-example}.

Recently in \cite{CarmonaEtAl19}, Carmona and Fern\'{a}ndez-S\'{a}nchez delved into the nature of Poincaré half-maps for planar linear differential systems and developed an integral characterization for these maps. The use of this new characterization in \cite{CarmonaEtAl19b} allowed a more straightforward and concise proof for the Lum-Chua conjecture, bypassing the usual extensive case-by-case approach. Applying the same idea,  \cite{Carmona2022-mc} revealed that the differential system \eqref{s1} is constrained to a maximum of one limit cycle, provided there is no sliding set. More recently, by combining this approach with an extended version of Khovanski\u{\i}’s theory, \cite{Carmona2023} established a finite upper bound for the number of limit cycles that can occur in systems of type~\eqref{s1}. This technique was also employed in \cite{CFSN-center} to solve the center problem for system~\eqref{s1}.

This paper aims to systematize the analysis performed in \cite{Carmona2023} by outlining a methodology for obtaining upper bounds on the number of limit cycles in any explicit example of a differential system of type \eqref{s1}. This methodology will be detailed in Section \ref{sec:metho}. Subsequently, in Section \ref{sec:example}, we will apply this methodology to demonstrate the following result:

\begin{theorem}\label{main}
The piecewise linear differential system \eqref{s1}, with the parameter values given either by \eqref{HY-example}, \eqref{FPTT-example}, or \eqref{GRS-example}, has exactly three limit cycles.
\end{theorem}

\noindent Finally, in Section \ref{ff-case}, the methodology will be used to derive an algebraic criterion for bounding crossing limit cycles in the focus-focus case. 

The reader may notice, as pointed out in Remark~\ref{anotherproof}, that the algebraic criterion which will be provided in Section \ref{ff-case} could also be used to prove Theorem~\ref{main}, and may reasonably ask why we did not present the criterion first, instead of analyzing the examples separately. Our choice to treat the examples individually is motivated by several reasons. First, the methodology we develop is not restricted to the focus-focus case, and the reasoning employed in Section~\ref{sec:example}, besides clearly illustrating its application, is also valuable for studying other configurations. Second, the analysis of specific examples offers a more didactic and intuitive grasp of the proof of the algebraic criterion, which might otherwise seem somewhat abstract or opaque. Finally, these explicit examples, especially the one introduced by Huan and Yang in \eqref{HY-example}, are of importance in the literature and merit detailed examination.

\section{Methodology}\label{sec:metho}

In this section, we establish a methodology for constraining the number of limit cycles of each differential system of type \eqref{s1}.  From now on, we consider system \eqref{s1} with all the parameters fixed. 

We recall that $a_{12}^{L}a_{12}^{R}\leq0$ avoids the existence of periodic solutions and, therefore, there are no limit cycles. Thus, in what follows we will deal with the nontrivial case $a_{12}^{L}a_{12}^{R}>0$.

Subsections \ref{sec:cf}-\ref{sec:clc} detail the step-by-step procedure of this approach, while Subsection \ref{outline} presents an overview of the methodology.

\subsection{Canonical Form}\label{sec:cf}
The procedure starts by transforming the given system \eqref{s1} into the Li\'{e}nard canonical form given in \cite[Proposition 3.1]{FreireEtAl12}. Since $a_{12}^{L}a_{12}^{R}>0$,  there exists  a homeomorphism  that preserves the separation line $\Sigma=\{(x,y)\in\R^2:\,x=0\}$  and writes the differential system \eqref{s1} as:
\begin{equation}\label{cf}
\left\{\begin{array}{l}
\dot x= T_L x-y\\
\dot y= D_L x-a_L
\end{array}\right.\quad \text{for}\quad x< 0,
\quad 
\left\{\begin{array}{l}
\dot x= T_R x-y+b^*\\
\dot y= D_R x-a_R
\end{array}\right.\quad \text{for}\quad x> 0,
\end{equation}
where $T_L,$ $T_R$ and $D_L,$ $D_R$ are, respectively, the traces and the determinants of the matrices $A_L$ and $A_R$,
\begin{equation*}\label{coef}
a_{L}=a_{12}^{L}b_2^{L}-a_{22}^{L}b_1^{L},\quad a_{R}=\dfrac{a_{12}^L}{a_{12}^R}(a_{12}^{R}b_2^{R}-a_{22}^{R}b_1^{R}), \quad\text{and}\quad b^*=a_{12}^Lb_1^R/a_{12}^R-b_1^L.
\end{equation*}
 
\begin{remark}\label{rekb}
Notice that, since system \eqref{cf} comes from the given system \eqref{s1}, all the parameters in \eqref{cf} are fixed. In the following sections, we will perturbe the parameter $b^*$, that is why we only added the star symbol $^*$ in this parameter.
\end{remark}

\subsection{Half-Poincar\'{e} maps}

In what follows, we will analyze the following 1-parameter family of piecewise linear systems:
\begin{equation}\label{1-cf}
\left\{\begin{array}{l}
\dot x= T_L x-y\\
\dot y= D_L x-a_L
\end{array}\right.\quad \text{for}\quad x< 0,
\quad 
\left\{\begin{array}{l}
\dot x= T_R x-y+b\\
\dot y= D_R x-a_R
\end{array}\right.\quad \text{for}\quad x> 0,
\end{equation}
for $b$ near $b^*$. Our approach consists in limiting the amount of limit cycles present in system \eqref{1-cf} for any $b$ sufficiently close to $b^*$, including the case when $b = b^*$.

In order to apply the ideas developed in \cite{Carmona2023}, we introduce two Poincaré Half-Maps defined on $\Sigma$. These are the \textit{Forward Poincaré Half-Map} $y_L: I_L \subset [0, +\infty) \rightarrow (-\infty, 0]$ and the \textit{Backward Poincaré Half-Map} $y_R^b: I_R^b \subset [b, +\infty) \rightarrow (-\infty, b]$, which, from \cite[Theorems 1 and 2]{Carmona2022-mc},  are well-defined if, and only if, the following set of conditions hold:
\begin{equation*}\label{def.cond}
\begin{cases}
a_L\leq 0\,\text{ and }\, 4D_L-T_L^2>0,\, \text{ or }\, a_L>0;\\
a_R\geq 0\,\text{ and }\, 4D_R-T_R^2>0,\, \text{ or }\, a_R<0.
\end{cases}
 \end{equation*}
 We stress that if the above conditions do not hold, then system \eqref{1-cf} does not have limit cycles.
In short, the forward flow of \eqref{1-cf} sends a point $(0, y_0)$ to $(0, y_L(y_0))$, while the backward flow of \eqref{1-cf} sends a point $(0, y_0)$ to $(0, y_R^b(y_0))$.
Clearly, the differential system \eqref{1-cf} restricted to $x\leq 0$ defines $y_L$, while the differential system of \eqref{1-cf} restricted to $x\geq 0$ defines $y_R^b$. It is important to note that if $y_R: I_R \subset [0, +\infty) \rightarrow (-\infty, 0]$ denotes the Backward Poincaré Half-Map of \eqref{1-cf} for $b = 0$ (that is,  $y_R=y_R^0$ and $I_R=I_R^0$),  then $y_R^b(y_0) = y_R(y_0 - b) + b$ and $I_R^b = I_R + b$.  The integral characterization for the maps $y_L$ and $y_R$ has been developed in \cite{CarmonaEtAl19}. We refer to \cite[Theorems 1 and 2]{Carmona2022-mc} for a concise description of this characterization. We will discuss some consequences of this characterization in the sequel.

Consider the polynomials
\begin{equation}\label{eq:poly}
W_L(y)=D_Ly^2-a_LT_Ly+a_L^2\quad\text{and}\quad W_R(y)=D_Ry^2-a_RT_Ry+a_R^2.
\end{equation}
As can be seen in \cite[Theorems 1 and 2]{Carmona2022-mc}, by denoting $\operatorname{ch}(\cdot)$ stands for the convex hull of a set, these polynomials satisfy $W_L(y)>0$ for $y \in \operatorname{ch}(I_L\cup y_L(I_L))\setminus\{0\}$ and  $W_R(y)>0$ for $y \in \operatorname{ch}(I_R\cup y_R(I_R))\setminus\{0\}$, and encode essential aspects of the maps $y_L$ and $y_R$. In fact, they provide the  following cubic vector fields
\begin{equation*}\label{dy_1L}
X_L(y_0,y_1)=-\big(y_1W_L(y_0) ,y_0 W_L(y_1)\big) \quad\text{and}\quad X_R(y_0,y_1)=-\big(y_1W_R(y_0) ,y_0 W_R(y_1)\big),
\end{equation*}
for which the graphs of $y_L$ and $y_R$, respectively, correspond to specific segments of orbits located in the closure $\ov{Q}$ of the open fourth quadrant 
\[Q:=\{(y_0,y_1):\,y_0>0\,\text{ and }\, y_1<0\}.\] Additionally, the maps $y_L(y_0)$ and $y_R(y_0)$ satisfy, respectively, the following differential equations:
 \begin{equation}
\label{eq:odeL}
\dfrac{d y_1}{d y_0}= \dfrac{y_0W_L(y_1)}{y_1W_L(y_0)} \quad \text{and}\quad  \dfrac{d y_1}{d y_0}= \dfrac{y_0W_R(y_1)}{y_1W_R(y_0)}.
\end{equation}

\begin{remark}\label{ILLR}
The polynomials $W_L$ and $W_R$ also give crucial properties of $I_L$ and $I_R$, intervals of definition of $y_L$ and $y_R$, respectively, and their images. Indeed, from \cite[Theorem 1]{Carmona2022-mc}) we have $I_L=[\lambda_L,\mu_L)$ and $y_L(I_L)=(\mu_L^1,y_L(\lambda_L)]$ satisfying:
\begin{itemize}
\item  $\mu_L$ corresponds to the smallest strictly positive root of $W_L$,  if it exists; if not $\mu_L=+\infty$.
\item $\lambda_L>0$ if, and only if, $a_L<0$, $4D_L-T_L^2>0$, and $T_L<0$.
\item  $\mu_L^1$ correspond to the largest strictly negative root of $W_L$, if it exists; if not $\mu_L^1=-\infty$.
\item $y_L(\lambda_L)<0$ if, and only if, $a_L<0$, $4D_L-T_L^2>0$, and $T_L>0$.
\item $\lambda_L>0$ implies $y_L(\lambda_L)=0$, while $y_L(\lambda_L)<0$ implies  $\lambda_L=0$.
\end{itemize}
Analogously, from \cite[Theorems 2]{Carmona2022-mc}, we have $I_R=[\lambda_R,\mu_R)$ and $y_R(I_R)=(\mu_R^1,y_R(\lambda_R)]$ satisfying:
\begin{itemize}
\item $\mu_R$ corresponds to the smallest strictly positive root of $W_R$, 
if it exists; if not $\mu_R=+\infty$.

\item $\lambda_R>0$  if, and only if, $a_R>0$, $4D_R-T_R^2>0$, and $T_R>0$.

\item $\mu_R^1$ correspond to the largest strictly negative root of $W_R$, if it exists; if not  $\mu_R^1=-\infty$.

\item $y_R(\lambda_R)<0$  if, and only if,  $a_R>0$, $4D_R-T_R^2>0$, and $T_R<0$.
\item$\lambda_R>0$ implies $y_R(\lambda_R)=0$, while $y_R(\lambda_R)<0$ implies $\lambda_R=0$.
\end{itemize}
\end{remark}

\subsection{Displacement function}

The importance of introducing the Poincaré half-maps above lies in their role in defining the following displacement function 
\[
\delta_b(y_0)=y_R(y_0-b)+b-y_L(y_0),\quad y_0\in I_b:=I_L\cap(I_R+b),
\]
which is crucial for analyzing the crossing periodic solutions of  \eqref{1-cf}. In fact, crossing periodic solutions, limit cycles, and hyperbolic limit cycles of  \eqref{1-cf} correspond bijectively to zeros, isolated zeros, and simple zeros, respectively, of $\delta_b$ in $\operatorname{int}(I_b)$. Obviously, if $\operatorname{int}(I_b)=\emptyset$, system \eqref{1-cf} has no crossing periodic solutions. Thus, in the search of limit cycles, we must assume that $\operatorname{int}(I_b)\neq \emptyset$.

\subsection{Bounding the number of hyperbolic crossing limit cycles} 
The next reasoning  will allow us to bound the number of hyperbolic crossing limit cycles of \eqref{1-cf}, equivalently, the number of simple zeros of $\delta_b$, for $b$ near $b^*$.

Let $y_0^*\in \operatorname{int}(I_b)$ satisfy $\delta_b(y_0^*)=0$. Then, from \eqref{eq:odeL}, we have 
\[
\delta_b'(y_0^*)=\dfrac{(y_0^*-y_1^*)}{y_1^*(y_1^*-b)W_L(y_0^*)W_R(y_0^*-b)}F_b(y_0^*,y_1^*),
\]
where $y_1^*=y_R(y_0^*-b)+b=y_L(y_0^*)<\min(0,b),$
\[
F_b(y_0,y_1)=m_0 + m_1 (y_0 + y_1) + m_2 y_0 y_1 + m_3 (y_0^2 + y_1^2) + 
 m_4 (y_0 y_1^2 + y_0^2 y_1) + m_5 y_0^2 y_1^2,
\]
and $m_i,$ $i=1,\ldots,5,$ are given by:
\begin{equation}\label{conic-coef}
\begin{aligned}
m_0 =& a_L^2 (a_R^2 b + b^3 D_R + a_R b^2 T_R),\\
m_1 = &-a_L^2 b (2 b D_R + a_R T_R) ,\\
m_2 = &-a_R^2 (b D_L + a_L T_L) + b D_R (3 a_L^2 - b^2 D_L + a_L b T_L) + 
    a_R (a_L^2 - b^2 D_L) T_R,\\
m_3 =& a_L^2 b D_R ,\\
m_4 = &a_R^2 D_L - D_R (a_L^2 - b^2 D_L + a_L b T_L) + a_R b D_L T_R ,\\
m_5 = & a_L D_R T_L - a_R D_L T_R-b D_L D_R .
\end{aligned}
\end{equation}

The following remark is a key observation to limit the amount of hyperbolic crossing limit cycles of \eqref{1-cf}.

\begin{remark}\label{key}
The number of simple zeros of $\delta_b$ corresponds to the number of isolated intersection points between the graphs of the Poincar\'{e} half-maps $y_L(y_0)$ and $y_R^b(y_0)$ for $y_0\in I_b$. We denote such graphs by:
\[
\widehat{\mathcal{O}}_{L}^b:=\{(y_0,y_L(y_0)):\,y_0\in \Int (I_b)\}\quad\text{and}\quad \widehat{\mathcal{O}}_{R}^b:=\{(y_0,y_R^b(y_0)):\,y_0\in \Int (I_b)\}.
\]
Let $U_b\subset Q$ be a simply connected domain containing either $\widehat{\mathcal{O}}_{L}^b$ or $\widehat{\mathcal{O}}_{R}^b$. Since $W_L(y_0^*)W_R(y_0^*-b)>0$ and $(y_0^*-y_1^*)y_1^*(y_1^*-b)>0$, we conclude that
$$
\sgn(\delta'_b(y_0^*))=\sgn(F_b(y_0^*,y_1^*)).
$$
This indicates that the curve $\gamma_b = F_b^{-1}(\{0\})\cap U_b$ divides the set $U_b$  into two open regions: one where $F_b(y_0,y_1)>0$ and another where $F_b(y_0,y_1)<0$. Consequently, the curve $\gamma_b$ separates the points in $U_b$ corresponding to attracting hyperbolic crossing limit cycles from those corresponding to repelling ones. Moreover, considering that two consecutive nested hyperbolic limit cycles must have opposite stability and taking into account that $\gamma_b$ is an algebraic curve and that $\widehat{\mathcal{O}}_{L}^b$ and $\widehat{\mathcal{O}}_{R}^b$ are analytic curves (because they are orbits of the polynomial vector fields) the number of hyperbolic limit cycles, and hence the number of simple zeros of the displacement function $\delta_b$, is constrained by the number of isolated intersection points between $\gamma_b$ and either the curve $\widehat{\mathcal{O}}_{L}^b$ or $\widehat{\mathcal{O}}_{R}^b$, increased by one.
\end{remark}

A useful tool for bounding the number of isolated intersection points between a smooth curve and  orbits of a vector field is the Khovanski\u{\i}'s Theorem (see \cite[Chapter II]{kho}). Here, we will use an extension of Khovanski\u{\i}’s Theorem to control the number of intersections between $\gamma_b$ and either the curve $\widehat{\mathcal{O}}_{L}^b$ or $\widehat{\mathcal{O}}_{R}^b$. This result employs the concept of a separating solution of vector fields defined on open simply connected subsets of $\R^2$, as defined in \cite[Definition 3]{Carmona2023}: an orbit $\mathcal{O}$ of a vector field vector field $X:U\to\R^2$, where $U$ is a open simply connected subset of $\R^2$, is said to be a {\bf separating solution} if it is either a cycle or a noncompact trajectory satisfying $(\ov{\mathcal{O}}\setminus\mathcal{O})\subset\partial U$. We notice that $\ov{\mathcal{O}}$ and $\partial U$ correspond, respectively, to the closure of $\mathcal{O}$ and the boundary of $U$, both with respect to $\R^2$. The extension of Khovanski\u{\i}’s Theorem is given as follows:

\begin{theorem}[{\cite[Theorem 4]{Carmona2023}}]\label{thm:count}
Consider a vector field $X:U\to\R^2$ defined on an open simply connected subset $U\subset\R^2.$ Let $\gamma\subset U$ be a smooth curve with at most $N$ noncompact (and any number of compact) connected components and having at most $k$ contact points with $X$. Then, there are at most $N+k$ isolated intersection points between $\gamma$ and any orbit of $X$ that is a separating solution.
\end{theorem}

We now introduce two fundamental conditions that enable the application of Theorem \ref{thm:count} to control the number of intersections between $\gamma_b$ and either the curve $\widehat{\mathcal{O}}_{L}^b$ or $\widehat{\mathcal{O}}_{R}^b$. First, define
\[
\mathcal{O}_{L}:=\{(y_0, y_L(y_0)) : y_0 \in \Int(I_L)\} \quad \text{and} \quad \mathcal{O}_{R}^b:=\{(y_0, y_R^b(y_0)) : y_0 \in \Int(I_R^b)\}.
\]
As noted in \cite{Carmona2023}, it is always possible to define a simply connected domain $U_b \subset Q$ and $\e > 0$ for which one of the following conditions holds: 
\begin{enumerate}[label = {$(S_L)$}]
\item\label{relSL} 
for all $b \in (b^* - \e, b^* + \e)$, $\widehat{\mathcal{O}}_{L}^b \subset U_b$ and $\mathcal{O}_{L} \cap U_b$ is a separating solution of $$\widehat{X}_L^b(y_0,y_1) := X_L(y_0,y_1)\big|_{U_b}.$$
\end{enumerate}
\begin{enumerate}[label = {$(S_R)$}]
\item\label{relSR} 
for all $b \in (b^* - \e, b^* + \e)$, $\widehat{\mathcal{O}}_{R}^b \subset U_b$ and $\mathcal{O}_{R}^b \cap U_b$ is a separating solution of $$\widehat{X}_R^b(y_0,y_1) := X_R(y_0-b,y_1-b)\big|_{U_b}.$$
\end{enumerate}
Some necessary information for constructing $U_b$ is provided in Remark \ref{ILLR}. 
For instance, in Section \ref{sec:example}, we will see that for the three worked examples it is enough to choose $U_b=Q$. However, for other cases it may be necessary to choose smaller regions.

According to Theorem \ref{thm:count}, we need to control the contact points between $\gamma_b$ and $\widehat X_L^b$ or $\widehat X_R^b$. For this sake, we define the following functions:
 \begin{equation}\label{GLGRdefini}
G_L^b(y_0,y_1)=\langle\nabla F_b(y_0,y_1),\widehat{X}_L^b(y_0,y_1)\rangle\quad \text{and}\quad G_R^b(y_0,y_1)=\langle\nabla F_b(y_0,y_1),\widehat{X}_R^b(y_0,y_1)\rangle.
 \end{equation}
The contact points between $\gamma_b$ and $\widehat X_L^b$ correspond to solutions in $U_b$ of the system: 
\begin{equation}\label{eq:polsystem}
 F_b(y_0,y_1)=0 \,\text{ and }\,  G_L^b(y_0,y_1)=0.
 \end{equation}
 Analogously, the contact points between $\gamma_b$ and $\widehat X_R^b$ correspond to solutions in $U_b$ of the system 
\begin{equation*}\label{eq:polsystemR}
 F_b(y_0,y_1)=0 \,\text{ and }\,  G_R^b(y_0,y_1)=0.
 \end{equation*}

\begin{remark}\label{rem:equiv}
We notice that if $(y_0,y_1)$ satisfies $F(y_0,y_1)=0$, then 
 \[
y_1 W_L(y_0) G_R^b(y_0,y_1)=(y_1-b) W_R(y_0-b)G_L^b(y_0,y_1).
 \]
 \end{remark}
 
In the sequel, in light of this Remark \ref{rem:equiv}, the relationship between $G_R^b$ and $G_L^b$ allow us to work only with system \eqref{eq:polsystem}. For this sake, denote $G_b=G_L^b$. One can see that
 \[
\begin{aligned}
G_b(y_0,y_1)=& n_1 (y_0 + y_1) + n_2  y_0 y_1 + n_3 (y_0^2 + y_1^2) + n_4 (y_0^2 y_1 + y_0 y_1^2) +
  n_5 (y_0^3 + y_1^3)\\
  & + n_6 y_0^2 y_1^2 + n_7 (y_0^3 y_1 + y_0 y_1^3) + 
 n_8 (y_0^3 y_1^2 + y_0^2 y_1^3) + n_9 y_0^3 y_1^3,
\end{aligned}
 \] 
 where $n_i,$ $i=1,\ldots 9,$ are given by
 \[
 \begin{aligned}
 n_1 =& a_L^4 b (2 b D_R + a_R T_R) ,\\
n_2 =& -2 a_L^3 b (2 a_L D_R + 2 b D_R T_L + a_R T_L T_R) ,\\
n_3 =& a_L^2 (a_R^2 (b D_L + a_L T_L) + b D_R (-3 a_L^2 + b^2 D_L - a_L b T_L) + 
     a_R (-a_L^2 T_R + b^2 D_L T_R)) ,\\
n_4 =& a_L (2 a_L^3 D_R + a_L^2 T_L (7 b D_R + a_R T_R) - 
     b D_L T_L (a_R^2 + b^2 D_R + a_R b T_R) \\&- 
     a_L (-b^2 D_R T_L^2 + a_R^2 (2 D_L + T_L^2) + a_R b D_L T_R)) ,\\
n_5 =& a_L^2 (-a_R^2 D_L + D_R (a_L^2 - b^2 D_L + a_L b T_L) - a_R b D_L T_R) ,\\
n_6 =& 2 (a_R^2 D_L (b D_L + 3 a_L T_L) + 
     D_R (b^3 D_L^2 - 2 a_L^3 T_L + a_L b^2 D_L T_L - 
        a_L^2 b (3 D_L + 2 T_L^2))\\& + 
     a_R D_L (-a_L^2 + b^2 D_L + 2 a_L b T_L) T_R) ,\\
n_7 =& a_L (a_R^2 D_L T_L + D_R T_L (-3 a_L^2 + b^2 D_L - a_L b T_L) + 
     a_R D_L (2 a_L + b T_L) T_R) ,\\
n_8 =& -3 a_R^2 D_L^2 + 
     D_R (-3 b^2 D_L^2 + a_L b D_L T_L + a_L^2 (3 D_L + 2 T_L^2)) - 
     a_R D_L (3 b D_L + 2 a_L T_L) T_R ,\\
n_9 =& 4 D_L (b D_L D_R - a_L D_R T_L + a_R D_L T_R) .
 \end{aligned}
 \]

Now, we proceed to simplify the polynomial system \eqref{eq:polsystem} constrained to $U_b\subset Q$. To do that, we consider the following map
 \[(Y_0,Y_1)=\phi(y_0,y_1):=(y_0+y_1,y_0 y_1),\] 
that corresponds to a diffeomorphism from $Q$ (the open fourth quadrant) onto $H=\{(Y_0,Y_1):Y_1<0\}$ (an open half-plane). This map transforms the polynomial system \eqref{eq:polsystem} restricted to $U_b$ into the following constraint polynomial system
\begin{equation}\label{eq:polsystem2}
\widetilde F_b(Y_0,Y_1)=0 \,\text{ and }\,  \widetilde G_b(Y_0,Y_1)=0,\,\, (Y_0,Y_1)\in \phi(U_b)\subset H,
 \end{equation}
where, now, $\widetilde F_b$  and $\widetilde G_b$ are, respectively, the following polynomial functions of degrees $2$ and $3$:
\begin{equation}\label{finalpolynomials}
\begin{aligned}
\widetilde F_b(Y_0,Y_1)=& m_0 + m_1 Y_0 + (m_2 - 2 m_3) Y_1+ m_3 Y_0^2  + m_4 Y_0 Y_1 + m_5 Y_1^2,\\
 \widetilde G_b(Y_0,Y_1)=&n_1 Y_0+ (n_2 - 2 n_3) Y_1 + n_3 Y_0^2 + (n_4 - 3 n_5) Y_0 Y_1+ (n_6 - 2 n_7) Y_1^2\\
 &+ n_5 Y_0^3 + n_7 Y_0^2 Y_1  + n_8 Y_0 Y_1^2 + n_9 Y_1^3.
\end{aligned}
\end{equation}

Thus, we have the following proposition.
\begin{proposition}\label{prop:hlc}
Let $U_b\subset Q$ be a simply connected domain and $\e>0$ for which one of the conditions \ref{relSL} or \ref{relSR} holds. Assume that there exists $b\in(b^*-\e,b^*+\e)$ for which the set $\widetilde \gamma_b=\widetilde F_b^{-1}(\{0\})=\phi(\gamma_b)$ has at most $N$ noncompact connected components  (and any number of compact ones) in $\phi(U_b)$ and  the constraint polynomial system \eqref{eq:polsystem2} has at most $k$ real isolated solutions (non-isolated solutions are allowed). Then, the amount of hyperbolic limit cycles of the differential system \eqref{1-cf} and, consequently,  the amount of simple zeros of the displacement function $\delta_b$ are bounded by $N+k+1$.
\end{proposition}
\begin{proof}
Let us prove this proposition assuming that condition \ref{relSL} holds. The reasoning for condition \ref{relSR} follows analogously.

First, suppose that all solutions of the constraint polynomial system \eqref{eq:polsystem2} are isolated. By reversing the transformation $\phi$, we find that $\gamma_b$ contains no more than $N$ noncompact connected components in $U_b$. Additionally, the polynomial system \eqref{eq:polsystem} has at most $k$ real solutions in $U_b$, which implies that $\gamma_b$ has no more than $k$ contact points with $\widehat X_L^b$ in $U_b$. Thus,  by Theorem \ref{thm:count}, the number of isolated intersections between $\gamma_b$ and $\mathcal{O}_{L}$, and therefore with $\widehat{\mathcal{O}}_{L}^b$, is bounded by $N+k$. Considering Remark \ref{key}, we conclude that the number of hyperbolic limit cycles of the differential system \eqref{1-cf}, and hence the number of simple zeros of $\delta_b$, is bounded by $N+k+1$.

Now, assume that the constraint polynomial system \eqref{eq:polsystem2} has non-isolated solutions. We analyze this case under two possible scenarios: (i) $\widetilde F_b$ is a factor of $\widetilde G_b$, or (ii) $\widetilde F_b$ and $\widetilde G_b$ share only a common linear factor.

In the first scenario, where $k = 0$, reversing the transformation $\phi$ reveals that $F_b$ is a factor of $G_b$. Consequently, from \eqref{GLGRdefini}, we deduce that $U_b\cap \gamma_b$ is an orbit of $\widehat X_L$. This implies that $\widehat{\mathcal{O}}_{L}^b$ do not intersect $\gamma_b$ at isolated points. Therefore, by Remark \ref{key}, the number of hyperbolic limit cycles of the differential system \eqref{1-cf}, and hence the number of simple zeros of $\delta_b$, is bounded by $1 \leq 1 + k+N$. 

In the second scenario, $\widetilde F_b$ and $\widetilde G_b$ share a common linear factor. This implies that $\widetilde F_b$ can be factored into two linear components, each one defining a straight line, denoted by $\ell_1$ and $\ell_2$, as its zero set. Consequently, the zero set of $\widetilde F_b$ is given by $\widetilde \gamma_b = \ell_1 \cup \ell_2$. In this scenario, one of these straight lines, say $\ell_1$, is also contained in the zero set $\widetilde G_b^{-1}(\{0\})$ and the other one, $\ell_2$, intersects $\widetilde G_b^{-1}(\{0\})$ at most in $k$ points. 
Reversing the transformation $\phi$, and  taking \eqref{GLGRdefini} and Remark \ref{rem:equiv} into account, we find that $ U_b\cap\phi^{-1}(\ell_1)$ is an orbit common to both vector fields $\widehat X_L$ and $\widehat X_R$. Accordingly, if $\widehat{\mathcal{O}}_{L}^b$ intersects $\phi^{-1}(\ell_1)$, then it must be entirely contained within $\phi^{-1}(\ell_1)$ and, consequently, also corresponds to an orbit of $\widehat{X}_R$ that, therefore, does not intersect $\widehat{\mathcal{O}}_R^b$ at isolated points, implying the absence of limit cycles. On the other hand, if $\widehat{\mathcal{O}}_L^b$ does not intersect $\phi^{-1}(\ell_1)$, we can adapt the set $U_b$ by defining $U_b' = \phi^{-1}(V_b)$, where $V_b$ is the connected component of $\phi(U_b) \setminus \ell_1$ that contains $\phi(\widehat{\mathcal{O}}_L^b)$. We notice that $U_b'$ still satisfies condition \ref{relSL}. In addition, since both $\ell_1$ and $\ell_2$ are straight lines, the number of noncompact connected components of $\widetilde \gamma_b$ within $V_b$ is still at most $N$. Consequently, the number of connected components of $\gamma_b$ within $U_b'$ is also at most $N$. Additionally, the polynomial system \eqref{eq:polsystem} has, now, at most $k$ real solutions in $U_b'$, implying that $U'_b\cap\gamma_b$ has no more than $k$ contact points with $\widehat X_L^b$. Hence, by Theorem \ref{thm:count}, the number of isolated intersections between $\gamma_b$ and $\widehat{\mathcal{O}}_L^b$ in $U_b'$ is bounded by $N + k$. Finally, by considering Remark \ref{key}, we conclude that the number of hyperbolic limit cycles of the differential system \eqref{1-cf}, and therefore the number of simple zeros of $\delta_b$, is bounded by $N + k + 1$.

\end{proof}

\subsection{Bounding the number of crossing limit cycles}\label{sec:clc}

With the purpose of bounding the number of crossing limit cycles of \eqref{1-cf}, we rely on the following technical lemma. It ensures that if any function in a one-parameter family of analytic functions has at most $N$ simple zeros, then, under a suitable monotonicity condition on the parameter, every function in the family has no more than $N$ isolated zeros.

\begin{lemma}[{\cite[Lemma 5]{Carmona2023}}]\label{lem:semi}
Let $I,J\subset\R$ be open intervals and consider a smooth function $\delta:I\times J\to \R$. Assume that
\begin{itemize}
\item[i.] for each $b\in J$, the function $\delta(\cdot,b)$ is analytic;
\item[ii.] there exists a natural number $N$ such that, for each $b\in J$, the number of simple zeros of the function $\delta(\cdot,b)$ does not exceed $N$; and 
\item[iii.] $\dfrac{\p \delta}{\p b}(u,b)>0$, for every $(u,b)\in I\times J$.
\end{itemize}
Then, for each $b\in J$, the function $\delta(\cdot,b)$ has at most $N$ isolated zeros.
\end{lemma}

Now, we are in position to extend Proposition \ref{prop:hlc} for crossing limit cycles in general, not only hyperbolic, as follows:
\begin{theorem}\label{thm:method}
Assume that either condition \ref{relSL} or condition \ref{relSR} holds. Suppose that there exists $\e>0$ such that, for each $b\in(b^*-\e, b^*+\e)$, $\widetilde \gamma_b=\widetilde F_b^{-1}(\{0\})$ has at most $N$ noncompact (and any number of compact) connected components in $\phi(U_b)$ and that the constraint polynomial system \eqref{eq:polsystem2} has at most $k$ isolated real solutions (non-isolated solutions are allowed). Then, the piecewise linear differential system \eqref{1-cf} has at most $N+k+1$ crossing limit cycles.
\end{theorem}

\begin{proof}[Proof of Theorem \ref{thm:method}] Let us prove this theorem assuming that condition \ref{relSL} holds. The reasoning for condition \ref{relSR} follows analogously.

In what follows, we use Lemma \ref{lem:semi} to establish that $\delta_{b^*}$ has no more than $N+k+1$ isolated zeros. From Proposition \ref{prop:hlc}, we know that for every $b \in (b^* - \e, b^* + \e)$, the function $\delta_b$ has at most $N+k+1$ simple zeros. Proceeding by contradiction, assume that $\delta_{b^*}$ has more than $N+k+1$ isolated zeros. Let $I \subset \Int(I_{b^*})$ be an open interval containing at least $N+k+2$ isolated zeros of $\delta_{b^*}$. Considering that $I_b = I_L \cap (I_R + b)$, let $J$ be an interval such that $b^* \in J \subset (b^* - \e, b^* + \e)$ and $I \subset \Int(I_b)$ for every $b \in J$.

Since for each $b \in J$, the restriction of the displacement function $\delta_b |_I$ has no more than $N+k+1$ simple zeros, and
\[
\dfrac{\partial}{\partial b} \delta_b(y_0) = -y_R'(y_0 - b) + 1 > 0,
\,\text{ for }\, y_0 \in I \,\text{ and }\, b \in J,
\]
it follows from Lemma \ref{lem:semi} that for each $b \in J$, $\delta_b |_I$ also has no more than $N+k+1$ isolated zeros. This contradicts the initial assumption that $\delta_{b^*}\in J$ has more than $N+k+1$ isolated zeros, thus completing the proof.

\end{proof}

\subsection{Outline of the methodology}\label{outline} 
We now summarize the preceding discussion by outlining the methodology used to bound the number of crossing limit cycles of system~\eqref{s1}, under the assumption that the following set of conditions holds:
\[
(H):\begin{cases}
 a_{12}^{L}a_{12}^{R}>0;\\
a_L\leq 0\,\text{ and }\, 4D_L-T_L^2>0,\, \text{ or }\, a_L>0;\\
a_R\geq 0\,\text{ and }\, 4D_R-T_R^2>0,\, \text{ or }\, a_R<0.
\end{cases}
\]
If condition (H) is not satisfied, then system~\eqref{s1} does not exhibit limit cycles.

\begin{itemize}
\item[{\bf Step 1:}] Transform system \eqref{s1} into the canonical form \eqref{cf} and consider the 1-parameter family of piecewise linear systems \eqref{1-cf}.
\item[{\bf Step 2:}] Establish a simply connected domain $U_b\subset Q$ and $\e>0$ for which either the condition \ref{relSL} or  \ref{relSR} holds.
\item[{\bf Step 3:}] Compute an upper limit $k$ for the number of real solutions of the constraint polynomial system \eqref{eq:polsystem2} for every $b\in(b^*-\e, b^*+\e)$.
\item[{\bf Step 4:}] Compute an upper limit $N$ for the number of noncompact connected components of $\widetilde \gamma_b=\widetilde F_b^{-1}(\{0\})$  in $\phi(U_b)$ for every $b\in(b^*-\e, b^*+\e)$.
\item[{\bf Step 5:}] Apply Theorem \ref{thm:method} to conclude that the piecewise linear differential system \eqref{s1} has no more than $N+k+1$ crossing limit cycles.
\end{itemize}

\subsection{Some remarks on the methodology}

\begin{remark}\label{improveU}
In Steps 3 and 4, in order to obtain smaller values for $N$ and $k$ and, consequently, improve the conclusion of Step 5, the set $U_b$ can eventually be modified and the value for $\e$ can be taken smaller as long as condition \ref{relSL} or \ref{relSR} remains true.
\end{remark}

\begin{remark}\label{conic}
Consider the coefficients $m_i$ given by \eqref{conic-coef}. If $m_3^2+m_4^2+m_5^2=0$, then $\widetilde F_b$ has degree $1$ and $\widetilde \gamma_b$ is a straight line. Otherwise, if $m_3^2+m_4^2+m_5^2\neq0$, then the set $\widetilde \gamma_b=\widetilde F_b^{-1}(\{0\})=\phi(\gamma_b)$ is a conic that can be analyzed by means of the discriminant 
\[
\Delta_b=-\det(H_{\widetilde F_b})=\left(D_R \left(a_L^2-a_L b T_L+b^2 D_L\right)+a_R^2 D_L+a_R b D_L T_R\right)^2-4 a_L^2 a_R^2 D_L D_R,
\]
where $H_{\widetilde F_b}$ is the Hessian of $\widetilde F_b$. Indeed: if $\Delta_b>0$, then $\widetilde \gamma_b$ is a hyperbola or two transversal intersecting straight lines; if $\Delta_b=0$, then $\widetilde \gamma_b$ is either a parabola, or two parallel straight lines, which could be coincident or not, or the empty set; and if $\Delta_b<0$, then $\widetilde \gamma_b$ is an ellipse or a single point. Finally, we notice that $m_3^2+m_4^2+m_5^2=0$ implies that $\Delta_b=0$. Thus, $\Delta_b\neq0$ actually implies $\widetilde \gamma_b$ is a conic; and $\Delta_b=0$ implies that $\widetilde \gamma_b$ is either a parabola, or two parallel straight lines, or  a single straight line, or the empty set.
\end{remark}

\begin{remark}\label{hyphalfplane}
For $D_R\neq0$, the equation 
\begin{equation}\label{Fbeq0}
\widetilde F_b(Y_0,0)=0
\end{equation}
 has the following solutions
\[
Y_0=b+\dfrac{a_R}{2D_R}\left(T_R+\sqrt{T_R^2-4D_R}\right)\quad\text{and}\quad Y_0=b+\dfrac{a_R}{2D_R}\left(T_R-\sqrt{T_R^2-4D_R}\right).
\]
Thus, in the case that the right system is a non-boundary focus, that is $T_R^2-4D_R<0$ and $a_R\neq0$, the equation \eqref{Fbeq0} has no real solutions. 
\end{remark}

\begin{remark}\label{knownsolu}
For $D_L\neq0$, the point
\[
(\widehat Y_0,\widehat Y_1)=\left(\dfrac{a_L T_L}{D_L},\dfrac{a_L^2}{D_L}\right) 
\]
is a solution of the polynomial system $\widetilde F_b(Y_0,Y_1)=0$ and  $\widetilde G_b(Y_0,Y_1)=0$.
\end{remark}

\section{Proof of Theorem \ref{main}}\label{sec:example}

We dedicate this section to apply the methodology presented in the previous section to prove Theorem \ref{main}. More specifically, in Subsections \ref{hy-system},  \ref{fptt-system}, and \ref{grs-system} we apply such a methodology to show that system \eqref{s1} with the parameter values given by \eqref{HY-example}, \eqref{FPTT-example}, and \eqref{GRS-example}, respectively, do not have more than three limit cycles. Since it is already known the existence of at least three limit cycles for such systems, the proof of Theorem \ref{main} will follow.

As we will see, the application of the steps outlined above will be very similar among the three examples. In {\bf Step 2}, we will set  $U_b=Q$ for all three examples. Thus, in {\bf Step 4}, we will conclude that $N=1$ for each example as a consequence of Remark \ref{hyphalfplane}.

\subsection{Analysis of the Huan-Yang example}\label{hy-system}
As the {\bf Step 1}, we start by noticing that system \eqref{s1}, with the parameters values given by \eqref{HY-example}, can be transformed into the Li\'enard canonical form \eqref{cf} where 
\begin{equation}\label{coefHY}
\begin{aligned}
&a_{L}=-\dfrac{117}{200},\quad a_R=-\dfrac{2861}{5000},\quad T_L=-\dfrac{3}{10},\quad T_R=\dfrac{19}{250},\\ &D_L=\dfrac{117}{200},\quad D_R=\dfrac{2861}{250000},\quad \text{ and }\quad b^*:=\dfrac{9}{10}.
\end{aligned}
\end{equation}
We also consider the 1-parameter family of piecewise linear systems \eqref{1-cf} for $b$ near $b^*$.

Now, consider the polynomials \eqref{eq:poly}. Taking Remark \ref{ILLR} into account, we see that, in the present case, since $W_L$ and  $W_R$ has no real roots, the intervals of definition $I_L$ and $I_R$ as well as their images $y_L(I_L)$  and  $y_R(I_R)$ are unbounded with $y_L(y_0)$ and $y_R(y_0)$  tending to $-\infty$ as $y_0\to +\infty.$  Moreover, since $a_L<0$, $4D_L-T_L^2>0$, and $T_L<0$, there exists $\lambda_L>0$ such that $I_L=[\lambda_L,+\infty)$ and $y_L(\lambda_L)=0$. Finally, since $a_R<0$ and $T_R>0$, we get $I_R=[0,+\infty)$ and $y_R(0)=0$. Thus, in the present case, $I_b=[\max\{\lambda_L,b\},+\infty)$ and $y_L(I_L)=y_R(I_R)=(-\infty,0]$.

In {\bf Step 2}, we notice that, for
\[
U_b=Q=\{(y_0,y_1)\in  \R^2:\, y_0>0,y_1<0\},
\]
the condition \ref{relSL} is verified for any $\e>0$. Indeed,  $\mathcal{O}_{b}\subset Q$ and, since $\mathcal{O}_{L}\cap Q=\mathcal{O}_{L}$ and $(\ov{\mathcal{O}_L}\setminus\mathcal{O}_L)=\{(\lambda_L,0)\}\subset\partial Q$, we get that $\mathcal{O}_{L}\cap Q$ is a separating solution of the restricted vector field $X_L\big|_{Q}.$ Notice in addition that $\phi(Q)=H=\{(Y_0,Y_1):Y_1<0\}.$

Now, we compute the polynomials $\widetilde F_b$  and $\widetilde G_b$, given by \eqref{finalpolynomials}, as
\[
\begin{aligned}
\widetilde F_b(Y_0,Y_1)=&C_1\Big[  117 b \left(100 b^2-380 b+2861\right)+2340 b (19-10 b) Y_0\\
&-20 \left(1000 b^3-4100 b^2+28025 b+10806\right) Y_1
+11700 b Y_0^2\\
&+500 \left(40 b^2-164 b+1121\right) Y_0 Y_1
+2000 (41-10 b) Y_1^2\Big],\\
 \widetilde G_b(Y_0,Y_1)=& C_2\Big[ 
 13689 b (10 b-19) Y_0 
 -234 \left(1000 b^3-3500 b^2+26885 b+10806\right) Y_1 \\
& +117 \left(1000 b^3-4100 b^2+26855 b+10806\right) Y_0^2 \\
 &+15 \left(-4000 b^3+39800 b^2-164360 b+175371\right) Y_0 Y_1\\
& +200 \left(2000 b^3-7600 b^2+51250 b+48021\right) Y_1^2\\
 &-2925 \left(40 b^2-164 b+1121\right) Y_0^3
 +300 \left(200 b^2-820 b+2407\right) Y_0^2 Y_1\\
& -3000 \left(200 b^2-780 b+5441\right) Y_0 Y_1^2 +80000 (10 b-41) Y_1^3\Big],
\end{aligned}
\]
where $C_1$ and $C_2$ are positive constant factors.

In {\bf Step 3}, by taking $b=b^*$ and computing the resultant between $F_{b^*}(Y_0,Y_1)$ and $G_{b^*}(Y_0,Y_1)$ with respect to the variable $Y_0$, we get
\[
\begin{aligned}
R(Y_1):=&\textrm{Res}_{Y_0}(F_{b^*},G_{b^*})\\
=&-C_3(200 Y_1-117) \Big(1601361472000000 Y_1^5+6760062816990000 Y_1^4\\
&-6284181440066400 Y_1^3+3572696139179193 Y_1^2-7126051666209372 Y_1\\
&+2258133778849572\Big),
\end{aligned}
\]
where $C_3$ is a positive constant factor.
By means of Descartes' Rule of Signs, one can see that $R$ has exactly one strictly negative root and, in addition, the other $5$ roots are strictly positive or complex. This implies that the polynomial system 
\begin{equation}\label{sistestrela}
\widetilde F_{b^*}(Y_0,Y_1)=0 \,\text{ and }\,  \widetilde G_{b^*}(Y_0,Y_1)=0
\end{equation}
has $6$ finite solutions, of which only one is contained in the open half-plane $H=\phi(Q)$ and the other $5$ solutions are contained in the open set $\R^2\setminus \ov{H}$. Now, by B\'{e}zout's theorem, the polynomial system \eqref{eq:polsystem2} has at most 6 solutions. Thus, by taking $\e>0$ sufficiently small, one has that, for each $b\in(b^*-\e,b^*+\e)$, the constraint polynomial system \eqref{eq:polsystem2} is a perturbation of \eqref{sistestrela} restricted to $H$ and, therefore, has at most one solution (in $H$). Then, $k=1$.
 
In {\bf Step 4}, in light of Remark \ref{conic}, in order to analyze the set $\widetilde \gamma_b=\widetilde F_b^{-1}(\{0\})$, we compute the discriminant:
\[
\Delta_b=C_4 \left(8000 b^4-65600 b^3+601600 b^2-1915192 b+6283205\right),
\]
where $C_4$ is a positive constant factor. We notice that for $b=b^*$, $\Delta_{b^*}= C_4 25021273/5>0$ and, therefore, by taking $\e>0$ smaller, one has that, for each $b\in(b^*-\e,b^*+\e)$, $\Delta_b>0$. It implies that, for each $b\in(b^*-\e,b^*+\e)$, the curve $\widetilde \gamma_b=\widetilde F_b^{-1}(\{0\})$ is a hyperbola (possibly degenerate). Thus, from Remark \ref{hyphalfplane}, $\widetilde \gamma_b$ has a single noncompact connected component in $H$ and, then, $N=1$.
 
Finally, in {\bf Step 5}, Applying Theorem \ref{thm:method}, we conclude that the piecewise linear differential system \eqref{s1}-\eqref{HY-example} has at most and, therefore, exactly $N+k+1=3$ crossing limit cycles.

\subsection{Analysis of the example provided by Freire et al.}\label{fptt-system}
In {\bf Step 1}, we transform \eqref{s1}, with parameter values  given in \eqref{FPTT-example} into the Li\'enard canonical form \eqref{cf} where 
\begin{equation}\label{coefHY}
\begin{aligned}
& a_{L}=\dfrac{65}{64},\quad a_R=\dfrac{96440395023695996806}{95571015330487000887},\quad T_L=-\dfrac{1}{4},\quad T_R=\dfrac{3276710}{13106841},\\
& D_L=\dfrac{65}{64},\quad D_R=\dfrac{174473488105306}{171789280999281},\quad\text{and}\quad b^*:=-\dfrac{521068}{1045519}.
\end{aligned}
\end{equation}
We also consider the 1-parameter family of piecewise linear systems \eqref{1-cf} for $b$ near $b^*$.

Now, consider the polynomials \eqref{eq:poly}. Taking Remark \ref{ILLR} into account, we see that, in the present case, analogous to  Huan-Yang example, the polynomials $W_L$ and  $W_R$ has no real roots, thus the intervals of definition $I_L$ and $I_R$ as well as their images $y_L(I_L)$  and  $y_R(I_R)$ are unbounded with $y_L(y_0)$ and $y_R(y_0)$  tending to $-\infty$ as $y_0\to +\infty.$ Moreover, since $a_L>0$ and $T_L<0$, we get $I_L=[0,+\infty)$ and $y_L(0)=0$. Finally, since $a_R>0$, $4D_R-T_R^2>0$, and $T_R>0$, there exists $\lambda_R>0$ such that $I_R=[\lambda_R,+\infty)$ and $y_R(\lambda_R)=0$.  Thus, in the present case, $y_L(I_L)=y_R(I_R)=(-\infty,0]$ and $I_b=[\max\{0,\lambda_R+b\},+\infty)$.

Hence, in {\bf Step 2}, the condition \ref{relSL} is verified for any $\e>0$ provided that
\[
U_b=Q=\{(y_0,y_1)\in  \R^2:\, y_0>0,y_1<0\}.
\]


Now, we have to compute the polynomials $\widetilde F_b$  and $\widetilde G_b$, which in the present case have cumbersome  expressions, so we shall omit them here.

In the {\bf Step 3}, by taking $b=b^*$, we computing the resultant between $F_{b^*}(Y_0,Y_1)$ and $G_{b^*}(Y_0,Y_1)$ with respect to the variable $Y_0$, $R(Y_1):=\textrm{Res}_{Y_0}(F_{b^*},G_{b^*})$. Again, the expression for $R(Y_1)$ is too large, so we will omit it here. In this case, in order to prove that $R(Y_1)$ has at most one strictly negative root, the Descartes' Rule of Signs does not work. However, the Sturm's theorem (see, for instance, \cite{SB02}) can be applied and it implies that $R(Y_1)$ has at exactly one strictly negative root. In addition, that the other $5$ roots are strictly positive or complex. This implies that the polynomial system 
\[
\widetilde F_{b^*}(Y_0,Y_1)=0 \,\text{ and }\,  \widetilde G_{b^*}(Y_0,Y_1)=0
\] 
has $6$ finite solutions, of which only one is contained in the open half-plane $H=\phi(Q)$ and the other $5$ solutions are contained in the open set $\R^2\setminus \ov{H}$. Thus, proceeding analogous to the Huan-Yang example, we obtain that $k=1$.

In {\bf Step 4}, in order to analyze the conic $\widetilde \gamma_b=\widetilde F_b^{-1}(\{0\})$ provided by the quadratic equation $F_b(Y_0,Y_1)=0,$ we compute the discriminant with respect to $Y_0$ of the homogeneous part of degree $2$ of $F_b$, yielding a positive discriminant for every $b$ sufficiently close to $b^*$. Thus, analogous to the Huan-Yang example we conclude that,  for each $b\in(b^*-\e,b^*+\e)$, the curve $\widetilde \gamma_b=\widetilde F_b^{-1}(\{0\})$ is a hyperbola. From Remark \ref{hyphalfplane}, $\widetilde \gamma_b$ has a single noncompact connected component in $H$ and, then, $N=1$.

Finally, in the {\bf Step 5}, Applying Theorem \ref{thm:method}, we conclude that the piecewise linear differential system \eqref{s1}-\eqref{FPTT-example} has at most and, therefore, exactly $N+k+1=3$ crossing limit cycles.

\subsection{Analysis of the example provided by  Gasull et al.}\label{grs-system}

In {\bf Step 1}, we transform system \eqref{s1}, with parameter values  given in \eqref{GRS-example}, into the Li\'enard canonical form \eqref{cf} where 
\begin{equation}\label{coefGRS}
\begin{aligned}
&a_{L}=-\dfrac{637}{625},\quad a_R=-\dfrac{60809}{8000},\quad T_L=-\dfrac{2}{5},\quad T_R=\dfrac{3}{4},\\
& D_L=\dfrac{26}{25},\quad D_R=\dfrac{73}{64},\quad\text{and}\quad b^*:=\dfrac{231}{200}.
\end{aligned}
\end{equation}
We also consider the 1-parameter family of piecewise linear systems \eqref{1-cf} for $b$ near $b^*$.

Notice that the sign of each coefficient in \eqref{coefGRS} coincides with the sign of each coefficient in \eqref{coefHY} for the Huan-Yang example. Thus, in {\bf Step 2}, analogous to the analysis performed in the Huan-Yang example, the condition \ref{relSL} is verified for any $\e>0$ provided that
\[
U_b=Q=\{(y_0,y_1)\in  \R^2:\, y_0>0,y_1<0\}.
\]

Now, we compute the polynomials $\widetilde F_b$  and $\widetilde G_b$, given by \eqref{finalpolynomials}, as
\[
\begin{aligned}
\widetilde F_b(Y_0,Y_1)=&C_1\Big[
31213 b \left(1000000 b^2-4998000 b+50653897\right)\\
&-62426000 b (1000 b-2499) Y_0
31250 \left(1000000 b^2-5390000 b+49655081\right) Y_0 Y_1\\
&-250 \left(125000000 b^3-673750000 b^2+6206885125 b+3106051249\right) Y_1\\
&+31213000000 b Y_0^2-312500000 (100 b-539) Y_1^2\Big],
\\
 \widetilde G_b(Y_0,Y_1)=& C_2\Big[
 7794010952 b (1000 b-2499) Y_0\\
& -62426 \left(125000000 b^3-575750000 b^2+5961983125 b+3106051249\right) Y_1\\
&+31213 \left(125000000 b^3-673750000 b^2+5957181125 b+3106051249\right) Y_0^2\\
&
-6125 \left(250000000 b^3-3258500000 b^2+18032110250 b-25418184099\right) Y_0 Y_1\\
&
+62500 \left(125000000 b^3-624750000 b^2+5693071125 b+6885054778\right) Y_1^2\\
&
-3901625 \left(1000000 b^2-5390000 b+49655081\right) Y_0^3\\
&
+1531250 \left(1000000 b^2-5390000 b+22187641\right) Y_0^2 Y_1\\
&
-3906250 \left(3000000 b^2-15386000 b+144739483\right) Y_0 Y_1^2\\
&
+156250000000 (100 b-539) Y_1^3\Big],
\end{aligned}
\]
where $C_1$ and $C_2$ are positive constant factors.

In the {\bf Step 3}, by taking $b=b^*$ and computing the resultant between $F_{b^*}(Y_0,Y_1)$ and $G_{b^*}(Y_0,Y_1)$ with respect to the variable $Y_0$, we get
\[
\begin{aligned}
R(&Y_1):=\textrm{Res}_{Y_0}(F_{b^*},G_{b^*})\\
=&-C_3(31250 Y_1-31213)\\ 
& \cdot \Big(103119356365203857421875000000000 Y_1^5
+742832966116277404785156250000000 Y_1^4\\ 
&
-1198304836247530573146057128906250 Y_1^3
+1124855330748568052975913837890625 Y_1^2\\ 
&
-3787786101124879021445139913050000 Y_1
+2038344678891458976535215823461357\Big),
\end{aligned}
\]
where $C_3$ is a positive constant factor. The very same analysis performed in the Huan-Yang example also applies here, establishing the existence of $\e>0$ for which the constraint polynomial system \eqref{eq:polsystem2}, for each $b\in(b^*-\e,b^*+\e)$, has at most one solution, providing, then, $k=1$.

In {\bf Step 4}, in order to analyze the conic $\widetilde \gamma_b=\widetilde F_b^{-1}(\{0\})$ provided by the quadratic equation $F_b(Y_0,Y_1)=0,$ we compute the discriminant with respect to $Y_0$ of the homogeneous part of degree $2$ of $F_b$, yielding
\[
\begin{aligned}
\Delta_b=&C_4 \Big(1000000000000 b^4-10780000000000 b^3+132357526000000 b^2\\
&-556816246140000 b+2465627069116561\Big),
\end{aligned}
\]
where $C_4$ is a positive constant factor. We notice that for $b=b^*$, $\Delta_{b^*}>0$.  Thus, analogous to the Huan-Yang example we conclude that,  for each $b\in(b^*-\e,b^*+\e)$, the curve $\widetilde \gamma_b=\widetilde F_b^{-1}(\{0\})$ is a hyperbola. Thus, from Remark \ref{hyphalfplane}, $\widetilde \gamma_b$ has a single noncompact connected component in $H$ and, then, $N=1$.
 
Finally, in the {\bf Step 5}, Applying Theorem \ref{thm:method}, we conclude that the piecewise linear differential system \eqref{s1}-\eqref{GRS-example} has at most and, therefore, exactly $N+k+1=3$ crossing limit cycles.

\section{Methodology applied for the focus-focus case}\label{ff-case}

In this section, we apply the methodology presented in Section \ref{sec:metho} and explored in Section \ref{sec:example} to provide  algebraic criteria for bounding the number of crossing limit cycles in the focus-focus case.

In what follows, let $R(Y_1)$ denote the resultant between $\widetilde F_{b^*}(Y_0,Y_1)$ and $\widetilde G_{b^*}(Y_0,Y_1)$ with respect to $Y_0$.

Our first result provides an upper bound for the number of limit cycles of \eqref{s1} by means of the number of roots of $R(Y_1)$ in $\mathbb{C}\setminus\{Y_1\in\R:Y_1\leq0\}$.

\begin{theorem}\label{focuscasegeneral}
Assume that $T_L^2-4D_L<0$ and $T_R^2-4D_R<0$. Let $\ell$ be the number of roots of $R(Y_1)$ in $\mathbb{C}\setminus\{Y_1\in\R:Y_1\leq0\}$. Then, system \eqref{s1} has no more than $8-\ell$ limit cycles.
\end{theorem}

\begin{proof}
In \textbf{Step 1}, we consider system \eqref{s1} in its Liénard canonical form \eqref{cf} and also the 1-parameter family of piecewise linear systems \eqref{1-cf}.

First, if $b^* = 0$, the main result from \cite{Carmona2022-mc} implies that system \eqref{s1} has at most one limit cycle. Moreover, if $a_L a_R = 0$, the number of limit cycles of \eqref{s1} is bounded by $2$. Indeed, from \cite{Carmona2021-ru},
\[
\sgn\left(\frac{d^2y_L}{dy_0^2}(y_0) \right)=-\sgn(a_L^2T_L) \quad \text{and}\quad \sgn\left(\frac{d^2y_R}{dy_0^2}(y_0) \right)=\sgn(a_R^2T_R),
\]
implying that the Poincaré half-return maps are convex. If $a_L a_R = 0$, one of these maps is a straight line, leading to at most two isolated intersections (see also \cite{LlibreEtAl15, Novaes2017}), which is less or equal to $8-\ell$ since $\ell\leq 6$. Therefore, we can assume $a_L a_R b^* \neq 0$.

Now, consider the polynomials in \eqref{eq:poly}. Since $T_L^2 - 4D_L < 0$ and $T_R^2 - 4D_R < 0$ by Remark \ref{ILLR}, it follows that $W_L$ and $W_R$ have no real roots. Consequently, the intervals of definition $I_L$ and $I_R$, as well as their images $y_L(I_L)$ and $y_R(I_R)$, are unbounded, with $y_L(y_0)$ and $y_R(y_0)$ tending to $-\infty$ as $y_0 \to +\infty$. Thus, in \textbf{Step 2}, we observe that for
\[
U_b = Q = \{(y_0, y_1) \in \mathbb{R}^2 : y_0 > 0, y_1 < 0\},
\]
condition \ref{relSL} is satisfied for any $\e > 0$. Indeed, $\widehat{\mathcal{O}}_{L}^b \subset Q$, and since $\mathcal{O}_{L} \cap Q = \mathcal{O}_{L}$ and $(\overline{\mathcal{O}_L} \setminus \mathcal{O}_L) = \{(\lambda_L, 0)\} \subset \partial Q$, we conclude that $\mathcal{O}_{L} \cap Q$ is a separating solution of the restricted vector field $X_L \big|_Q$. Additionally, note that $\phi(Q) = H = \{(Y_0, Y_1) : Y_1 < 0\}$.

In \textbf{Step 3}, by B\'ezout's Theorem, the polynomial system 
\[
\widetilde{F}_{b}(Y_0, Y_1) = 0 \quad \text{and} \quad \widetilde{G}_{b}(Y_0, Y_1) = 0
\]
has at most 6 solutions for any $b$. From hypothesis, the resultant  $R(Y_1)$ has $\ell$ roots (counting multiplicity)  in $\mathbb{C}\setminus\{Y_1\in\R:Y_1\leq0\}$. Since the polynomial function  $Y_0\mapsto\widetilde{F}_{b}(Y_0, Y_1)$ has constant degree for every $Y_1$, we have that the above polynomial system, for $b=b^*$, has $\ell$ solutions (counting multiplicity) contained in the open set $\mathbb{C}\setminus\{Y_1\in\R:Y_1\leq0\}$. Therefore, there exists $\e>0$ small such that, for each $b\in(b^*-\e,b^*+\e)$,  the above polynomial system has at least $\ell$ solutions (counting multiplicity) contained in the open set $\mathbb{C}\setminus\{Y_1\in\R:Y_1\leq0\}$. Therefore, the constraint polynomial system \eqref{eq:polsystem2} has at most $6-\ell$ solutions, that is $k\leq 6-\ell$.

In \textbf{Step 4}, by Remark \ref{conic}, $\widetilde{\gamma}_{b} = F_{b}^{-1}(\{0\})$ has at most two noncompact connected components in $\mathbb{R}^2$. According to Remark \ref{hyphalfplane}, since $T_R^2 - 4D_R < 0$ and $a_R \neq 0$, $\widetilde{\gamma}_{b}$ does not intersect the $Y_0$-axis. Furthermore, by Remark \ref{knownsolu}, $(\widehat{Y}_0, \widehat{Y}_1) \in \widetilde{\gamma}_b$ with $\widehat{Y}_1 > 0$ because $a_L \neq 0$ and $D_L > 0$. Thus, $\gamma_b$ has at most one connected component in $H$, implying $N \leq 1$.

Finally, in \textbf{Step 5}, by applying Theorem \ref{thm:method}, we conclude that system \eqref{s1} has no more than $N + k + 1 \leq 8-\ell$ limit cycles.
\end{proof}

By Remark \ref{knownsolu}, $\widehat{Y}_1$ is a root of $R_1(Y_1)$ with $\widehat{Y}_1 > 0$ because $a_L \neq 0$ and $D_L > 0$, thus, consequently, contained in $\mathbb{C}\setminus\{Y_1\in\R:Y_1\leq0\}$. Therefore, for each $b$, we have that $\ell\leq1$ and, then, the next result follows immediately from Theorem \ref{focuscasegeneral}.

\begin{corollary}\label{focus7}
If $T_L^2-4D_L<0$ and $T_R^2-4D_R<0$, then system \eqref{s1} has no more than $7$ limit cycles.
\end{corollary}

The upper bound provided by Corollary \ref{focus7} improves by one the result provided in \cite{Carmona2023}. The upper bound given by Theorem \ref{focuscasegeneral} and Corollary \ref{focus7} can be refined under certain conditions.

As a first refinement of Theorem \ref{focus7}, the following result provides an algebraic criterion to bound the number of limit cycles of \eqref{s1} in the focus-focus case, assuming that the 6th-degree coefficient of the resultant $R(Y_1)$ between $\widetilde F_{b^*}(Y_0,Y_1)$ and $\widetilde G_{b^*}(Y_0,Y_1)$ with respect to $Y_0$ is non-zero. Note that the degree of $R(Y_1)$ is at most $6$, and it is strictly less than $6$ if the polynomial system $\widetilde F_{b^*}(Y_0,Y_1) = 0$ and $\widetilde G_{b^*}(Y_0,Y_1) = 0$ has solutions at infinity.
\begin{theorem}\label{focuscase}
Assume that $T_L^2-4D_L<0$, $T_R^2-4D_R<0$, and that $R(Y_1)$ has non-vanishing $6$th degree coefficient. Let $k$ be the number of non-positive roots of $R(Y_1)$ (counting multiplicity). Then, system \eqref{s1} has no more than $k+2$ limit cycles.
\end{theorem}

\begin{proof}
With similar reasoning as in the proof of Theorem \ref{focuscasegeneral}, we can assume $a_L a_R b^* \neq 0$. Additionally, \textbf{Step 1}, \textbf{Step 2}, and \textbf{Step 4} are identical to those in Theorem \ref{focuscasegeneral}, which implies that $N \leq 1$.

We proceed by improving the estimate for $k$ in \textbf{Step 3}. Let $b = b^*$, and consider the resultant $R(Y_1)$ between $F_{b^*}(Y_0, Y_1)$ and $G_{b^*}(Y_0, Y_1)$ with respect to the variable $Y_0$. By hypothesis, $R$ is a polynomial of degree 6 with a non-zero 6th-degree coefficient and has $k$ non-positive roots (counting multiplicity). Since the polynomial function  $Y_0\mapsto\widetilde{F}_{b}(Y_0, Y_1)$ has constant degree for every $Y_1$,  this implies that the polynomial system 
\[
\widetilde{F}_{b^*}(Y_0, Y_1) = 0 \quad \text{and} \quad \widetilde{G}_{b^*}(Y_0, Y_1) = 0
\]
has 6 finite solutions (counting multiplicity), of which $k$ are in the closed half-plane $\overline{H}$ and, since there are no solutions at the infinity, the remaining $6 - k$ are either in the open set $\mathbb{R}^2 \setminus \overline{H}$ or are complex. By B\'ezout's theorem, the polynomial system \eqref{eq:polsystem2} has at most 6 solutions. Thus, by choosing $\e > 0$ sufficiently small (see Remark \ref{improveU}), we conclude that for each $b \in (b^* - \e, b^* + \e)$, the polynomial system \eqref{eq:polsystem2} represents a small perturbation of the above polynomial system and, therefore, has $6 - k$ solutions (counting multiplicity) in $\mathbb{R}^2 \setminus \overline{H}$ or as complex roots. This implies that the constraint polynomial system \eqref{eq:polsystem2} has at most $k$ solutions in $H$.

Finally, in \textbf{Step 5}, by applying Theorem \ref{thm:method}, we conclude that system \eqref{s1} has no more than $N + k + 1 \leq k + 2$ limit cycles.
\end{proof}

\begin{remark}\label{anotherproof}
System \eqref{s1}, with parameter values given either by \eqref{HY-example}, \eqref{GRS-example}, or \eqref{FPTT-example}, satisfies the assumptions of Theorem \ref{focuscase}. In addition, $k=1$ for all the three system and, therefore, Theorem \ref{main} also follows as a consequence of Theorem \ref{focuscase}.
\end{remark}

As a second refinement of Theorem \ref{focus7}, our final result for the focus-focus case provides an improved upper bound compared to Theorem \ref{focuscase}, under the following condition:
\[
P=\begin{cases}
\Delta_{b^*} < 0, \,\, \text{or} \\
\Delta_{b^*} = 0,\,\, b^* = \dfrac{a_L D_R T_L - a_R D_L T_R}{2D_L D_R},\,\, \text{and}\,\, a_L a_R < 0,
\end{cases}
\]
where $\Delta_b$ is defined in Remark \ref{conic}.

\begin{theorem}\label{focuscase2}
Assume that $T_L^2-4D_L<0$ and $T_R^2-4D_R<0$. If $P$ holds,
then system \eqref{s1} has no more than one limit cycle.
\end{theorem}
\begin{proof}
First, if $b^* = 0$, the main result from \cite{Carmona2022-mc} implies that system \eqref{s1} can have no more than one limit cycle. Additionally, it is straightforward to see that condition $P$ cannot hold if $a_L a_R = 0$. Therefore, we can assume $a_L a_R b^* \neq 0$. Additionally, \textbf{Step 1} and \textbf{Step 2} are identical to those in Theorem \ref{focuscasegeneral}.

Next, we claim that condition $P$ implies the existence of $\e > 0$ such that $\Delta_b \leq 0$ for all $b \in (b^* - \e, b^* + \e)$. Indeed, if $\Delta_{b^*} < 0$, the continuity of $\Delta_b$ with respect to $b$ ensures that the claim holds. Conversely, if $\Delta_{b^*} = 0$, then by condition $P$,
\[
b^* = \dfrac{a_L D_R T_L - a_R D_L T_R}{2D_L D_R} \implies \dfrac{\partial \Delta_b}{\partial b}\Big|_{b = b^*} = 0 \quad \text{and} \quad \dfrac{\partial^2 \Delta_b}{\partial b^2}\Big|_{b = b^*} = 8a_L a_R \sqrt{D_L^3 D_R^3} < 0.
\]
This implies that for a sufficiently small $\e > 0$, $\Delta_b < 0$ for all $b \in (b^* - \e, b^* + \e) \setminus \{ b^* \}$.

Moreover, by Remark \ref{hyphalfplane}, the curve $\widetilde{\gamma}_{b}$ does not intersect the $Y_0$-axis since $T_R^2 - 4D_R < 0$ and $a_R \neq 0$. Thus, it cannot contain a straight line. Indeed, if it did, that line would have to be parallel to the $Y_0$-axis, which would imply that $F_{b}^{-1}(Y_0, 0) = a_L^2 D_R b Y_0^2$ is zero for every $Y_0$, a contradiction since $a_L^2 D_R b \neq 0$.

Consequently, by Remark \ref{conic}, $\widetilde\gamma_{b} = F_{b}^{-1}({0})$ must be either an ellipse, a single point, a parabola, or the empty set. Furthermore, since $a_L \neq 0$ and $D_L > 0$, by Remark \ref{knownsolu}, we have $(\widehat{Y}_0, \widehat{Y}_1) \in \widetilde\gamma_b$ with $\widehat{Y}_1 > 0$. Thus, $\gamma_b \cap \overline{H} = \emptyset$ for all $b \in (b^* - \e, b^* + \e)$.

Therefore, in \textbf{Step 3} and \textbf{Step 4}, we get, for each $b \in (b^* - \e, b^* + \e)$, that $N = 0$ and that the constraint polynomial system \eqref{eq:polsystem2} has no solutions, i.e., $k = 0$. Hence, in \textbf{Step 5}, by applying Theorem \ref{thm:method}, we conclude that system \eqref{s1} has at most $N + k + 1 = 1$ limit cycle.
\end{proof}

\begin{remark}
Notice that, for $a_L a_R \neq 0$, the negation of condition $P$ is expressed as follows:
\[
\neg P=\begin{cases}
\Delta_{b^*}>0,\,\,\text{ or }\\
\Delta_{b^*}=0,\,\, b^*\neq\dfrac{a_L D_R T_L-a_R D_L T_R}{2D_L D_R},\,\,\text{ or }\\
\Delta_{b^*}=0,\,\, b^*=\dfrac{a_L D_R T_L-a_R D_L T_R}{2D_L D_R},\,\,\text{and}\,\, a_L a_R>0.
\end{cases}
\]
In the proof of Theorems \ref{focuscasegeneral} and \ref{focuscase}, it can be easily concluded that condition $\neg P$ implies that for each $\e > 0$, there exists $b \in (b^* - \e, b^* + \e)$ such that $\Delta_b > 0$ and, consequently, $N = 1$ since $\widetilde{\gamma}_{b} = F_{b}^{-1}(\{0\})$ would be a hyperbola that does not intersect the $Y_0$-axis. Therefore, for $U_b = Q = \{(y_0, y_1) \in \mathbb{R}^2 : y_0 > 0, y_1 < 0\}$, the estimate provided by Theorems \ref{focuscasegeneral} and \ref{focuscase} cannot be further improved using the proposed methodology when condition~$P$ is not satisfied.
\end{remark}

\section*{Acknowledgements}
VC and FFS are partially supported by the Ministerio de Ciencia e Innovaci\'on, Plan Nacional I+D+I cofinanced with FEDER funds, in the frame of the project PID2021-123200NB-I00, the Consejer\'{i}a de Educaci\'{o}n y Ciencia de la Junta de Andaluc\'{i}a TIC-0130. DDN is partially supported by S\~{a}o Paulo Research Foundation (FAPESP) grants 2018/13481-0, and by Conselho Nacional de Desenvolvimento Cient\'{i}fico e Tecnol\'{o}gico (CNPq) grant 309110/2021-1.

\bibliographystyle{abbrv}
\bibliography{references.bib}

\end{document}